\documentclass[11pt]{article}
\usepackage[intlimits]{amsmath}
\usepackage{amsfonts, amsthm}
\usepackage{array}
\usepackage{amscd}
\usepackage{enumerate}
\usepackage{multicol}
\usepackage{amssymb}
\usepackage{amsfonts}
\usepackage{enumerate}
\usepackage[T1]{fontenc}
\usepackage[cp1250]{inputenc}
\newtheorem{thm}{Theorem}
\newtheorem{lem}{Lemma}
\newtheorem{rem}{Remark}
\newtheorem{exm}{Example}
\newtheorem{col}{Corollary}

\newcommand{\D }{\Delta }

\title{Oscillatory properties of solutions of the fourth order difference equations with quasidifferences }

\author{
Robert Jankowski\footnote{Lodz University of Technology, Poland, email: rjjankowski@math.uwb.edu.pl}
\and
Ewa Schmeidel \footnote{University of Bialystok, Poland, email: eschmeidel@math.uwb.edu.pl}
\and
Joanna Zonenberg\footnote{University of Bialystok, Poland, email: jzonenberg@math.uwb.edu.pl}
}  
\date{}
\begin{document}
\maketitle
\noindent
{\bf Abstract.} A class of fourth--order neutral type difference equations with quasidifferences and deviating arguments is considered. Our approach is based on studying the considered equation as a system of a four--dimensional difference system. The sufficient conditions under which the considered equation has no quickly oscillatory solutions are given. Finally, the sufficient conditions under which the equation is almost oscillatory are presented.\\
{\small{\bf Keywords:} Fourth--order difference equation, neutral type, quickly oscillatory solutions, almost oscillatory.}\\
{\small{\bf AMS Subject classification:} 39A21, 39A10}

\section{Introduction}
In this article, we consider a class of fourth--order nonlinear difference equations of the form:
\begin{equation}
\label{eq:1}
 \D \left\lbrace a_n\left[ \D \left( b_n\left(\D ( c_n (\D (x_n + p_nx_{n-\delta}))^{\gamma} )\right) ^{\beta }\right) \right]^{\alpha} \right\rbrace  +d_nf(x_{n-\tau})=0
\end{equation}
where 
$\alpha$, $\beta$ and $\gamma $ are the ratios of odd positive integers, integers $\tau, \delta$ are deviating arguments, $\tau \neq \min\lbrace -4,\delta-4\rbrace$. Moreover, $(p_n)$ is a real sequence, $(d_n)$ is of one sign, and $(a_n), (b_n), (c_n)$ are positive real sequences defined for $n\in \mathbb{N}_0=\{n_0, n_0+1, \dots \}$, $n_0 \in \mathbb{N}=\lbrace 0, 1, 2, \dots \rbrace$, $n_0\geq  \max\lbrace 1, \delta, \tau \rbrace $. Function $f: \mathbb{R} \to \mathbb{R}$, where $\mathbb{R}$ denotes the set of real numbers. Here $\D$ is the forward difference operator defined for any real sequence $(y_n)$ by $\D y_n=y_{n+1}-y_n$. If $\tau> \min\lbrace -4,\delta-4\rbrace$, then obviously there exists a solution of equation \eqref{eq:1} that can be found recursively for given initial conditions. If $\tau< \min\lbrace -4,\delta-4 \rbrace$, then in order to ensure the existence of a unique solution of equation \eqref{eq:1} for given initial conditions, we assume in this paper that the function $f$ is invertible. 

In the last few years, great attention has been paid to the study of fourth--order nonlinear difference equations. It is interesting to extend the known oscillation criteria for a larger class of fourth order nonlinear difference equations with quasidifferences. The oscillatory and asymptotic properties of fourth order nonlinear difference equations were investigated, among many others, by: Agarwal, Grace and Manojlović \cite{AGM, AM_2009}, Do\v{s}l\'a and Krej\v{c}ov\'a \cite{DK2012}, Migda, Musielak, Popenda, Schmeidel and Szmanda \cite{MS2004,MMS2004,MMS2006,PJ,bib8,bib9,bib10}, Smith and Taylor \cite{ST}, and Arockiasamy, Dhanasekaran, Graef, Pandian and Thandapani \cite{TG,TPDG,TA}, and references therein.

The background for difference equations can be found in the well known monographs, see, for example: Agarwal, Bohner, Grace and O'Regan \cite{monografiaAB}, Agarwal \cite{monografiaA}, and Kelly and Peterson \cite{monografiaP}.

Set 
\begin{equation}\label{z}
z_n=x_n + p_nx_{n-\delta}.
\end{equation}
The sequence $(z_n)$ is called companion sequence of a sequence $(x_n)$ relative to $(p_n)$.

By a solution of equation \eqref{eq:1} we mean a real sequence $(x_n)$ satisfying equation \eqref{eq:1} for $n\in \mathbb{N}_0$. A nontrivial solution $(x_n)$ of \eqref{eq:1} is called nonoscillatory if it is either eventually positive or eventually negative, and it is otherwise oscillatory. Equation \eqref{eq:1} is called almost oscillatory if all its solutions are oscillatory or
\[
\lim\limits_{n \to \infty}x_n =0.
\]

A solution  $(x_n)$ of equation \eqref{eq:1} is called quickly oscillatory if
$$x_n=(-1)^nq_n, \textrm{ where } q_n \textrm{ is of one sign for } n\in \mathbb{N}_0.$$

If $p_n\equiv 0$ and $f(x) = x^\lambda$, then equation \eqref{eq:1} takes the following form
\[
\D(a_n(\D b_n(\D c_n(\D x_n)^{\gamma })^{\beta })^{\alpha})+d_n x_{n+\tau}^\lambda=0.
\]
Since in our consideration $\tau \in \mathbb{Z}$ we see that this special case of equation \eqref{eq:1} with negative $\tau$ was studied in \cite{DK2012}. The results presented in this paper have improved and generalized those obtained by Do\v{s}l\'a and Krej\v{c}ov\'a.  

We consider \eqref{eq:1} as a four--dimensional system. If
\[
y_n=c_n(\D z_n)^{\gamma}, \qquad w_n=b_n(\D y_n)^{\beta}, \qquad t_n=a_n(\D w_n)^{\alpha},
\]
then equation \eqref{eq:1} can be written as the nonlinear system
\begin{equation}
\label{eq:S}
\left\{ \begin{array}{l}
        \D (x_n + p_nx_{n-\delta})=C_n y_n^{\frac{1}{\gamma }},\\
        \D y_n=B_nw_n^{\frac{1}{\beta }},\\
        \D w_n=A_nt_n^{\frac{1}{\alpha }},\\
        \D t_n=-d_nf(x_{n-\tau}), \end{array} \right. 
\end{equation}
where
\begin{equation}\label{e4}
A_n=a_n^{-\frac{1}{\alpha }}, \qquad B_n=b_n^{-\frac{1}{\beta}}, \qquad C_n=c_n^{-\frac{1}{\gamma }}.
\end{equation}

In this paper we study the properties of solutions of equation \eqref{eq:1}. Firstly, we show the influence of the deviating arguments $\delta$ and  $\tau $ on the existence of quickly oscillatory solutions of \eqref{eq:1}. Next, some monotonic properties of solution of the considered equation written as system \eqref{eq:S} are given. Finally, the sufficient conditions under which equation \eqref{eq:1} is almost oscillatory are given. The results are illustrated by examples.

\section{Some basic lemmas}
In 2005, Migda and Migda presented the following result which will be used in the sequel (see \cite{MM2005}, Lemma 1). 
\begin{lem}\label{MM1}
Let $(x_n)$, $(p_n)$ be a real sequences and $(z_n)$ be a sequence define by \eqref{z}, for $n\geq \delta$. Assume that $(x_n)$ is bounded, $\lim\limits_{n \to \infty}z_n = l \in \mathbb{R}$ ,  $\lim\limits_{n \to \infty}p_n = p \in \mathbb{R}$. If $\vert p \vert\neq 1$, then $(x_n)$ is convergent and $\lim\limits_{n \to \infty}x_n = \frac{l}{1+p}$.
\end{lem}
The next lemma is a simple generalization of Lemma 2 from Migda and Migda paper \cite{MM2005}.
\begin{lem}\label{MM2}
Assume that $x \colon \mathbb{N} \to \mathbb{R}$ and
\begin{equation}\label{lp}
\lim\limits_{n \to \infty}p_n = p \,\, \mbox{ where } \,\, \vert p \vert < 1.
\end{equation}
If sequence $(z_n)$ defined by \eqref{z} is bounded, then sequence $(x_n)$ is bounded too.
\end{lem}
\begin{proof}
Let $L > 0$ be such that $\vert z_n \vert \leq L$ for $n \geq 1$. Set $P\colon =\frac{1+ \vert p \vert}{2}$. From \eqref{lp}, we get $0< P < 1$, and there exists $n_1 \in \mathbb{N}$ such that $\vert p_n \vert \leq P$, for $n\geq n_1$. Set  
\begin{equation}\label{K}
K= \max\lbrace \vert x_{n_1}\vert, \vert x_{n_1+1}\vert, \dots  \vert x_{n_1+\delta+1}\vert \rbrace.
\end{equation}
Let $n \geq n_1$ be an arbitrary integer. Then there exists $m \in \mathbb{N}$ such that $n_1 \leq n - m\delta \leq n_1+\delta$. From \eqref{z}, we have 
\[
\vert x_n \vert \leq \vert z_n \vert + P \vert x_{n-\delta} \vert \leq L + P \vert x_{n-\delta} \vert.
\]
Analogously
\[
\vert x_{n-\delta} \vert \leq L + P \vert x_{n-2\delta} \vert,
\]
and
\[
\vert x_n \vert \leq L + P L + P^2 \vert x_{n-2\delta} \vert.
\]
After $m$ steps, we obtain
\[
\vert x_n \vert \leq L + P L + P^2 L + \dots + P^{m-1}L + P^m \vert x_{n-m\delta} \vert.
\]
Using \eqref{K}, we have $\vert x_{n-m\delta} \vert \leq K$. Since $0<P<1$, we obtain $P^m \vert x_{n-m\delta} \vert < K$. From the above,
\[
\vert x_n \vert \leq L(1+ P  + P^2 + \dots + P^{m-1}) + K \leq K + \frac{L}{1-P}.
\]
This completes the proof. 
\end{proof}
Note that, putting $p_n \equiv p$ in the above Lemma, we get Lemma 2 from paper \cite{MM2005}.
\section{Existence of quickly oscillatory solutions}

In this section necessary conditions for the existence of quickly oscillatory solution of equation \eqref{eq:1} are presented.

\begin{thm}\label{Th.1}
Assume that $p_n\geq 0$, $d_n >0$ for $n \in N_0$, $\delta $ is even, and 
\begin{equation}\label{e1}
xf(x)>0, \mbox{ for }x\neq 0.
\end{equation}
If $\tau$ is even, then equation \eqref{eq:1} has no quickly oscillatory solutions with positive even terms. If $\tau$ is odd, then equation \eqref{eq:1} has no quickly oscillatory solutions with positive odd terms.
\end{thm}

\begin{proof}
Let 
\begin{equation}\label{e3}
x_n=(-1)^nq_n \mbox{ be a quickly oscillatory solution of \eqref{eq:1}.}
\end{equation}
Assume also that $(q_n)$ is a positive sequence. 
Since $\delta $ is even, we get
\begin{align*}
\Delta z_n=\Delta(x_n+p_nx_{n-\delta})= x_{n+1}+p_{n+1}x_{n+1-\delta}-x_n-p_nx_{n-\delta}\\
=(-1)^{n+1}(q_{n+1}+q_n+ p_{n+1}q_{n-\delta+1}+p_nq_{n-\delta}).
\end{align*}
Hence $(\Delta z_n)$ is a quickly oscillatory sequence and we can write $\Delta z_n=(-1)^{n+1}s_n$ where $s_n = q_{n+1}+q_n+ p_{n+1}q_{n-\delta+1}+p_nq_{n-\delta} >0$. From the first equation of system $ \eqref{eq:S}$ we have
$$ y_n= \left(\frac{\D z_n}{C_n }\right)^{\gamma} =(-1)^{n+1}r_n, $$
where $r_n=(\frac{s_n}{C_n})^{\gamma} >0$. From the second equation of $ \eqref{eq:S}$ we get
$$ w_n= \left( \frac{\D y_n}{B_n }\right) ^{\beta } =(-1)^{n}l_n, $$
where $l_n=\left( \frac{r_{n+1}}{B_n}+\frac{r_n}{B_n} \right)^{\beta} >0$. Repeating argument, we get from the third equation of $ \eqref{eq:S}$ the following equality
$$ t_n= \left( \frac{\D w_n}{A_n }\right) ^{\alpha } =(-1)^{n+1}g_n, $$
where $g_n=\left( \frac{l_{n+1}}{A_n}+\frac{l_n}{A_n} \right)^{\alpha} >0$. Consequently, from the fourth equation we have
$$ \D t_n=(-1)^{n}(g_{n+1}+g_n)=-d_nf(x_{n-\tau}).$$
Hence
\begin{equation}\label{e2}
(-1)^{n+1}(g_{n+1}+g_n)=d_nf(x_{n-\tau}).
\end{equation}
By \eqref{e1}, we have $x_{n-\tau}f(x_{n-\tau})>0$. From \eqref{e3} we see that $x_{n-\tau}=(-1)^{n-\tau}q_{n-\tau}$. Since $\tau$ is even we get $f(x_{n-\tau})$ is positive for even $n$. The left--hand side of equality \eqref{e2} is negative for even $n$, whereas the right--hand side is positive and vice versa. This contradiction ended the proof in the case of even $\tau$.

For odd $\tau$ the proof is analogous and hence is omitted.
\end{proof}

\begin{rem}
Assume that $p_n\geq 0$, $d_n <0$ for $n \in N_0$, $\delta $ is even, and condition \eqref{e1} is satisfied. If $\tau$ is even, then equation \eqref{eq:1} has no quickly oscillatory solutions with positive odd terms. If $\tau$ is odd, then equation \eqref{eq:1} has no quickly oscillatory solutions with positive even terms.
\end{rem}

The following examples illustrate Theorem \ref{Th.1}. Here $\delta$ is even and equation of the form \eqref{eq:1} has a quickly oscillatory solution with positive even terms.

\begin{exm}
Consider the equation
\begin{equation}\label{eq:6}
\D ^2\left(\D ^2 \left(x_n+\frac{1}{2^n}x_{n-2\lambda }\right)\right)^{\beta }
+d_n\operatorname{sgn}(x_{n-\tau})=0,
\end{equation}
where $\tau$, $\lambda $ are some positive integers, and $\tau$ is odd,
\[
d_n=(2^{n+2}3^2+2^{2-2\lambda })^{\beta }
+2(2^{n+1}3^2+2^{2-2\lambda })^{\beta }+(2^{n}3^2+2^{2-2\lambda })^{\beta }>0,
\]
and $f(x)=\operatorname{sgn}x$. Equation \eqref{eq:6} has a quickly oscillatory solution $x_n=(-1)^n2^n$. 
\end{exm}

\begin{exm}
Consider the equation
\begin{equation}\label{eq:7}
\D ^2\left(\D ^2\left(x_n+\frac{1}{3^n}x_{n-{2\lambda }}\right)\right)^{\beta }+d_nx_{n-\tau}=0
\end{equation}
where $\lambda$ is some positive integer, $\tau $ is an odd integer and
\[ 
d_n=\left(4+\frac{4^2}{3^{n+4}}\right)^{\beta}+2\left(4+\frac{4^2}{3^{n+3}}\right)^{\beta}+\left(4+\frac{4^2}{3^{n+2}}\right)^{\beta}>0
\]
and $f(x)=x$. Equation \eqref{eq:7} has a quickly oscillatory solution $x_n=(-1)^n$.
\end{exm}

\section{Almost oscillatory property}

In this section we assume that there exists a finite limit of sequence $(p_n)$. After some lemmas concerning the behavior of solution of system \eqref{eq:S}, the sufficient conditions under which equation \eqref{eq:1} is almost oscillatory are presented.  

\begin{lem}
Assume that conditions \eqref{lp} and \eqref{e1} are satisfied. If $(x,y,w,t)$ is a solution of system \eqref{eq:S} with bounded first component and such that one of its components is of one sign, then one of the following two cases hold:
\begin{enumerate}
\item[1)]
all its components are of one sign for enough large $n$;
\item[2)] 
sequence $(y_n)$ is of one sign and $\lim\limits_{n \to \infty}x_n = 0$.
\end{enumerate}

\end{lem}
\begin{proof}
Firstly, we assume that sequence $(x_n)$ is positive. (For negative $(x_n)$ the proof is analogous and hence omitted.) Since $(d_n)$ is of one sign for $n\geq n_0$, and by \eqref{e1}, we have that $(\D t_n)$ is also of one sign. It means that if $\D t_n >0$, then sequence $(t_n)$ is increasing, and if $\D t_n<0$, then sequence $(t_n)$ is decreasing. In both cases sequence $(t_n)$ is of one sign for large $n$. This implies, by the third equation of system \eqref{eq:S}, that $\D w_n$ is of one sign for large $n$. Repeating the same arguments as before, we obtain that sequence $(w_n)$ is of one sign, and consequently $(y_n)$ is of one sign for large $n$. Hence, condition $1)$ of the thesis is fulfilled.

Next, we assume that $(y_n)$ is of one sign for $n \in \mathbb{N}_0$. By the first equation of system \eqref{eq:S}, we get that sequence $(z_n)$ defined by \eqref{z} is monotonic for large $n$. Hence $\lim\limits_{n \to \infty}z_n$ exists. Let us consider three possible cases:
\begin{enumerate}
\item[$(i)$]
\[
\lim\limits_{n \to \infty}z_n = \pm \infty,
\]
\item[$(ii)$]
\[
\lim\limits_{n \to \infty}z_n = l \neq 0,
\]
\item[$(iii)$]
\[
\lim\limits_{n \to \infty}z_n = 0.
\]
\end{enumerate}
Case $(i)$. If $(z_n)$ is unbounded, then $(x_n)$ is unbounded too. By assumption of the Lemma this case is excluded. \\
Case $(ii)$.  Since sequence $(z_n)$ is bounded, by condition \eqref{lp} and Lemma \ref{MM2}, sequence $(x_n)$ is bounded too.
By Lemma \ref{MM1}, we have $\lim\limits_{n \to \infty}x_n = \frac{l}{1+p} \neq 0$. It means that sequence $(x_n)$ is of one sign for large $n$.\\
Case $(iii)$.  By Lemmas \ref{MM1} and \ref{MM2},
we get $\lim\limits_{n \to \infty}x_n = 0$. It means that condition $2)$ of the thesis is satisfied.

Assuming that $(t_n)$ or $(w_n)$ is of one sign, on virtue of analogous arguments as the above, we also get the thesis. 
\end{proof}

\begin{col}\label{Col.1}
Assume that  \eqref{lp} and \eqref{e1} are satisfied. If $(x,y,w,t)$ is a solution of system \eqref{eq:S} with bounded first component and such that one of its components is of one sign, then there exists limit of sequence $(x_n)$ and exactly one of the following two cases are held:
\begin{enumerate}
\item[1)]
$\lim\limits_{n \to \infty}x_n \neq 0$, \\
and sequences $y$, $w$ and $t$ are monotonic for enough large $n$, or
\item[2)] 
sequence $(y_n)$ is of one sign and $\lim\limits_{n \to \infty}x_n = 0$.
\end{enumerate}

\end{col}

\begin{col}\label{C2}
Assume that conditions \eqref{lp} and \eqref{e1} are satisfied, and 
\begin{equation}\label{abc}
\sum\limits_{n=n_0}^{\infty} A_n= \sum\limits_{n=n_0}^{\infty}  B_n = \sum\limits_{n=n_0}^{\infty} C_n = \infty.
\end{equation} 
If $(x,y,w,t)$ is a solution of system \eqref{eq:S} such that 
\begin{equation}\label{lx}
\lim\limits_{n \to \infty}x_n \in \mathbb{R},
\end{equation}
then 
\[
\lim\limits_{n \to \infty}y_n =\lim\limits_{n \to \infty}w_n = \lim\limits_{n \to \infty}t_n =0.
\]
\end{col}
\begin{proof}
Since conditions \eqref{lp} and \eqref{lx} hold, $ \lim\limits_{n \to \infty}z_n$ is finite. From the first equation of system \eqref{eq:S}, we get 
\begin{equation}\label{es1}
z_n = z_{n_0} + \sum\limits_{i= n_0}^{n-1} C_i y_i^{\frac{1}{\gamma} }.
\end{equation}
Without loss the generality, for the sake of contradiction, assume that\\ $\lim\limits_{n \to \infty} y_n > 0$. Hence $\lim\limits_{n \to \infty} y_n^{\frac{1}{\gamma}}  > 0$. Since $C_n$ is positive, letting $n$ to infinity, the right hand side of equality \eqref{es1} tends to infinity whereas the left hand side has a finite limit. It follows that $\lim\limits_{n \to \infty} y_n =0$. In the similar manner the remaining part of the thesis is obtained. 
\end{proof}
\begin{thm}\label{Th.2}
If conditions \eqref{lp}, \eqref{e1} and \eqref{abc} are satisfied, $f\colon \mathbb{R} \to \mathbb{R}$ is continuous function,
and series
\begin{equation}\label{zs}
\sum_{i=1}^{\infty} d_i \mbox{ is divergent},
\end{equation} 
then any bounded solution of equation \eqref{eq:1} is almost oscillatory. 
\end{thm}
\begin{proof}
For the sake of contradiction assume that equation \eqref{eq:1} has a nonoscillatory bounded solution which does not approach zero. Without loss of generality we assume that $x_n > 0$ for large $n$. From the above and Corollary \ref{Col.1}, we have that finite limit $\lim\limits_{n \to \infty} x_n$ exists. Set $\lim\limits_{n \to \infty} x_n =c \in (0, \infty)$. So, by \eqref{e1}, we have $f(c)>0$.
Then there exists a positive integer $n_1$ such that $f(x_{n-\tau}) \geq \frac{f(c)}{2}$ for $n\geq n_1$. This implies, that:\\
if $(d_n)$ is a positive sequence, then
\[
\sum_{i=n_1}^{\infty} d_i f(x_{i-\tau}) \geq \frac{f(c)}{2}\sum_{i=n_1}^{\infty} d_i = + \infty;
\] 
if  $(d_n)$ is a negative sequence, then
\[
\sum_{i=n_1}^{\infty} d_i f(x_{i-\tau}) \leq \frac{f(c)}{2}\sum_{i=n_1}^{\infty} d_i = - \infty.
\] 
Summing the last equation of system \eqref{eq:S} from $1$ into $n-1$, we have
\[
t_n - t_1= -\sum_{i=1}^{n-1} d_if(x_{i-\tau}).
\]
By Corollary \ref{C2}, we get that sequence $(t_n)$ tends to zero as $n$ tends to $\infty$. Hence, letting $n$ to $\infty$, we obtain
\[
t_1= \sum_{i=1}^{\infty} d_i f(x_{i-\tau}).
\]
The left--hand side of the above equation is a constant whereas the right--hand is $-\infty$ or $+\infty$. This contradiction ended the proof.
\end{proof}
The following two examples show equations of type  \eqref{eq:1} the coefficients of which fulfill the assumptions of Theorem \ref{Th.2} and both of them have an almost oscillatory solution. The first one has an asymptotically zero solution, and the second one has an oscillatory solution. 
\begin{exm}
Consider the equation
\[
\D \left(n \D ^3 (x_n+\frac{1}{4}x_{n-2})\right)+(1-n)x_{n+3}=0.
\]
The above equation has solution $x_n=-\frac{1}{2^n}$ which tends to zero. 
\end{exm}
\begin{exm}
Consider the equation
\[
\D \left(n \D ^3 (x_n+\frac{1}{4}x_{n-2})\right)+10(2n+1)x_{n+3}=0.
\]
This equation has oscillatory solution $x_n=\frac{(-1)^{n}}{10}$. 
\end{exm}


{\small
\begin{thebibliography}{99}

\bibitem{monografiaA}
R. P. Agarwal, {\it Difference equations and inequalities. Theory, methods and applications}, 
Marcel Dekker, Inc., New York 1992.

\bibitem{monografiaAB}
R.P. Agarwal, M. Bohner, S.R. Grace and D. O'Regan, {\it Discrete Oscillation Theory}, Hindawi Publishing Corporation, New York,
2005.

\bibitem{AGM}
R.P. Agarwal, S.R Grace, J.V Manojlović \emph{On the oscillatory properties of certain fourth order nonlinear difference equations} \textit{ Math. Anal. Appl.} \textbf{ 322} 930--956, (2006).

\bibitem{AM_2009}
R.P. Agarwal, J.V. Manojlović, \emph{Asymptotic behavior of nonoscillaory solutions of fourth order nonlinear difference equations} \textit{Dyn. Contin. Discrete Impuls. Syst. Ser. A, Math. Anal.}  \textbf{16} 155--174, (2009).


\bibitem{DK2012}
Z. Do\v{s}l\'a, J. Krej\v{c}ov\'a,  \emph{Oscillation of a class of the fourth--order nonlinear difference equations}, \textit{Adv. Difference Equ.}, doi:10.1186/1687-1847-2012-99, (2012).

\bibitem{monografiaP}
W.G. Kelly, A.C. Peterson, {\it Difference equations}, Academic Press, Inc., Boston-San Diego 1991.

\bibitem{MM2005}
M. Migda, J. Migda, \emph{Asymptotic properties of solutions of second--order neutral difference equations}, {\it Nonlinear Anal.}, \textbf{63}  e789–-e799, (2005).

\bibitem{MS2004}
M. Migda, E. Schmeidel, \emph{Asymptotic properties of fourth order nonlinear difference equations}, \textit{Math. Comput. Modelling}, \textbf{39}  1203--1211, (2004).

\bibitem{MMS2004}
M. Migda, A. Musielak, E. Schmeidel, \emph{On a class of fourth order nonlinear difference equations}, \textit{Adv. Difference Equ.}, \textbf{1} 23--36, (2004).

\bibitem{MMS2006}
M. Migda, A. Musielak, E. Schmeidel, \emph{Oscillatory of fourth order nonlinear difference equations with quasidifferences},\textit{Opuscula Math.}, \textbf{26(2)} 371--380, (2006).

\bibitem{PJ} 
J. Popenda, E. Schmeidel, \emph{On the solution of fourth order difference equations}, \textit{ Rocky Mt. J. Math.}, \textbf{25(4)} 1485--1499, (1995).

\bibitem{bib8}
E. Schmeidel,  \emph{Oscillation and nonoscillation theorems for fourth order difference equations}, {\it Rocky Mt. J.Math.}, \textbf{33(3)} 1083--1094, (2003).

\bibitem{bib9}
E. Schmeidel,  \emph{Nonscillation and oscillation properties for fourth order nonlinear difference equations}, {\it New Progress in Difference Equations}, edited by B. Aulbach, S. Elaydi, G. Ladas, CRC, Boca Raton, FL, 531--538, (2004).

\bibitem{bib10}
E. Schmeidel, B. Szmanda,  \emph{Oscillatory and asymptotic behavior of certain difference equation}, {\it Nonlinear Anal.}, \textbf{47} 4731-4742, (2001).

\bibitem{ST}
B. Smith, W.E. Taylor, \emph{ Oscillatory and asymptotic behavior of certain fourth order difference equations}  \textit{Rocky Mt. J. Math.} \textbf{16(2)} 403--406, (1986).

\bibitem{TG}
E. Thandapani, J.R. Graef, \emph{ Oscillatory and asymptotic behavior of fourth order nonlinear delay difference equations} \textit{Fasc. Math.} \textbf{31}  23--36, (2001).
 
\bibitem{TPDG}
E. Thandapani, S. Pandian, R. Dhanasekaran, J. Graef,  \emph{Asymptotic results for a class of fourth order quasilinear difference equations} \textit{J. Differ. Equ. Appl.} \textbf{13(12)} 1085--1103, (2007).


\bibitem{TA}
E. Thandapani, I.M. Arockiasamy, \emph{On fourth order nonlinear oscillations of difference equations} \textit{Comput. Math. Appl.} \textbf{42} 357--368, (2001).

\end{thebibliography}
}
\end{document}